\documentclass[10pt,a4paper]{amsart}

\usepackage{amssymb,amsmath,amsfonts}

\usepackage{graphicx}

\usepackage{mathpazo}
\usepackage[hyperindex,pageanchor]{hyperref}

\numberwithin{equation}{section}
\newtheorem{theorem}{Theorem}[section]
\newtheorem{lemma}[theorem]{Lemma}
\newtheorem{proposition}[theorem]{Proposition}
\newtheorem{corollary}[theorem]{Corollary}

\theoremstyle{definition}
\newtheorem{definition}[theorem]{Definition}
\theoremstyle{remark}
\newtheorem{remark}[theorem]{Remark}

\newcommand{\Ass}{\operatorname{Ass}}

\newcommand{\grade}{\operatorname{grade}}
\newcommand{\height}{\operatorname{height}}

\newcommand{\Spec}{\operatorname{Spec}}

\newcommand{\cd}{\operatorname{cd}}

\newcommand{\Ext}{\operatorname{Ext}}
\newcommand{\End}{\operatorname{End}}
\newcommand{\Supp}{\operatorname{Supp}}

\newcommand{\Hom}{\operatorname{Hom}}

\newcommand{\Ann}{\operatorname{Ann}}
\newcommand{\Rad}{\operatorname{Rad}}

\newcommand{\depth}{\operatorname{depth}}

\begin{document}

\author[Eghbali ]{Majid Eghbali}
\author[Schenzel]{Peter Schenzel}

\title[Endomorphisms and local cohomology]
{On an endomorphism ring of local cohomology}

\address{Martin-Luther-Universit\"at Halle-Wittenberg, Institut f\"ur Informatik, D -- 06099 Halle (Saale), 
 Germany.}
\email{m.eghbali@yahoo.com}
\email{peter.schenzel@informatik.uni-halle.de}



\begin{abstract} Let $I$ be an ideal of a local ring $(R,\mathfrak m)$ 
with $d = \dim R.$ For the local cohomology module $H^i_I(R)$ it is a well-known fact that it vanishes for $i > d$ and is an Artinian 
$R$-module for $i = d.$ In the case that the Hartshorne-Lichtenbaum Vanishing 
Theorem fails, that is $H^d_I(R) \not= 0,$ we explore its fine structure. 
In particular, we investigate its endomorphism ring and related connectedness 
properties. In the case $R$ is complete we prove - as a technical tool - 
that $H^d_I(R) \simeq H^d_{\mathfrak m}(R/J)$ for a certain ideal $J 
\subset R.$ Thus, properties of $H^d_I(R)$ and its Matlis dual might be 
described in terms of the local cohomology supported in the maximal ideal. 
\end{abstract}

\maketitle

\section{Introduction}
Let $I \subset R$ denote an ideal of a local ring $(R, \mathfrak m).$ Let $M$ be a finitely generated 
$R$-module with $d = \dim M.$   For an integer $i \in \mathbb Z$ let $H^i_I(M)$ denote the $i$-th 
local cohomology module of $M$ with respect to $I$ as introduced by Grothendieck (see \cite{G} and \cite{Br-Sh}). 
Of a particular interest are the first non-vanishing (resp. the last non-vanishing) cohomological degree 
of the local cohomology modules $H^i_I(M),$ known as the grade $\grade (I,M)$ (resp. cohomological 
dimension $\cd (I,M)$). It is a well-known fact that 
\[
\grade (I,M) \leq \cd (I,M) \leq \dim M.
\]
In the case of $I = \mathfrak m$ it follows that $\cd (\mathfrak m, M) = \dim M$ 
(see \cite{G}). While for an arbitrary 
ideal $I \subset R$ the Hartshorne-Lichtenbaum Vanishing Theorem says that the following conditions are equivalent: 
\begin{itemize}
\item[(1)] $H^d_I(M) = 0.$ 
\item[(2)] $\dim \hat{R}/I\hat{R} +\mathfrak p > 0 \text{ for all }  \mathfrak p 
\in \Ass_{\hat{R}} \hat{M} \text{ such that } \dim \hat{R}/\mathfrak p = d.$
\end{itemize}
(see \cite{H} and \cite{Br-Sh}). Here $\hat{M}$ resp. $\hat{R}$ denotes the completion of $M$ resp. $R.$  
Moreover, it follows that $H^d_I(M)$ is an Artinian $R$-module (see 
\cite{H} and \cite{S1}).  Furthermore, there is an explicit description of the Artinian $R$-module $H^d_I(M)$ 
by its secondary decomposition and its attached prime ideals (see \cite[Section 3]{D-Sch}). 
Let $E_R(R/\mathfrak{m})$ denote the injective hull of the residue field.
Then $\Hom_R(H^d_I(M),E_R(R/\mathfrak{m}))$ is a finitely generated $\hat{R}$-module. One of 
our interest is to study the properties of it. 

In recent research there is an interest in endomorphism rings of certain local cohomology 
modules $H^i_I(R).$ This was done in the case of $i = \dim R$ and $I = \mathfrak m$ by Hochster 
and Huneke (see \cite{HH}) and in the case of $i = \height I$ and $R$ a Gorenstein ring (see \cite{Sch5} and the 
references there). Here we continue with the case of $i = \dim R$ and an arbitrary ideal $I \subset R.$ In particular 
we investigate the natural ring homomorphism 
\[
\hat{R} \to \Hom_{\hat{R}}(H^d_I(R), H^d_I(R)), \; d = \dim R.
\]
We describe its kernel, characterize when it is an isomorphism, prove that the endomorphism ring 
$\Hom_{\hat{R}}(H^d_I(R), H^d_I(R))$ is commutative and decide when it is a local Noetherian ring. Note that 
in the case of $I = \mathfrak{m}$ we recover results shown by Hochster and Huneke (see \cite{HH}). In fact, 
we use their result in our proof.

\begin{theorem} \label{0.1} Let $I$ denote an ideal of a complete local 
ring $(R,\mathfrak m).$ Let $J = \Ann_R H^d_I(R)$ where $d = \dim R.$ 
\begin{itemize} 
  \item[(a)] The endomorphism ring $\Hom_R(H^d_I(R), H^d_I(R))$ is a 
 commutative semi-local ring, finitely generated as $R$-module. 
  \item[(b)] The natural homomorphism $R/J \to \Hom_R(H^d_I(R), H^d_I(R))$ 
  is an isomorphism if and only if $R/J$ satisfies the condition $S_2.$ 
\end{itemize} 
\end{theorem}   

Moreover, we describe the  ideal $J$ explicitly in terms of $R$ and $I.$ 
Of course this is only of an interest whenever $H^d_I(R) \not= 0.$ That is, when the ideal $I \subset R$ 
does not satisfy condition (2) above. As an important step towards the proof of Theorem 
\ref{0.1} we prove the following result:

\begin{theorem} \label{1.1} Let $(R,\mathfrak{m})$ denote a complete local ring. Let $I \subset R$ denote 
an ideal. Let $M$ be a  finitely generated $R$-module  with $d = \dim M.$ Then there is an ideal 
$J \subset R$ such that 
\[
H^d_I(M) \simeq H^d_{\mathfrak{m}}(M/JM) \simeq H^d_{\mathfrak{m}}(M)/JH^d_{\mathfrak{m}}(M).
\]
The ideal $J$ is described explicitly.
\end{theorem}

\section{Auxiliary results}
Let $(R,\mathfrak{m})$ denote a local ring. Let  $E_R(R/\mathfrak{m})$ denote  the injective hull of the 
residue field $R/\mathfrak{m} = k.$ The Matlis duality functor $\Hom_R(\cdot, E_R(R/\mathfrak{m}))$ 
is denoted by $D(\cdot).$ 
 We need a Lemma concerning $R$-modules such that $\Supp_R M \subseteq \{\mathfrak{m}\},$ see e.g. 
 \cite[Lemma 2.1]{Sch6}.

\begin{lemma} \label{2.1} Let $M$ be an  $R$-module and $N$ an $\hat R$-module.
\begin{itemize}
\item[(a)] Suppose $\Supp_R M \subseteq V(\mathfrak m).$ Then $M$ admits a unique
$\hat R$-module structure compatible with its $R$-module structure such that the natural
map $M \otimes_R \hat R \to M$ is an isomorphism.
\item[(b)] The module $\Ext_R^i(M,N), i \in \mathbb Z,$ might be considered as an $\hat R$-module
such that there is a natural isomorphism $\Ext^i_R(M,N) \simeq \Ext^i_{\hat R}(M \otimes_R  \hat R, N).$
\item[(c)] The Matlis dual $D(M)$ admits a natural $\hat R$-module structure.
\end{itemize}
\end{lemma}

We need the definition of the canonical module of a finitely generated $R$-module $M.$ To this end 
let $(R,\mathfrak m)$ be the epimorphic image of a local Gorenstein ring $(S,\mathfrak{n})$ and 
$n = \dim  S.$ 

\begin{definition} \label{2.2}  (A)
For a finitely generated $R$-module $M$ we consider 
\[
K(M) = \Hom_R(H^d_{\mathfrak{m}}(M),E_R(R/\mathfrak m)),  \; d = \dim M.
\] 
Because $H^d_{\mathfrak{m}}(M)$ is an Artinian $R$-module it turns out that $K(M)$ is a 
 finitely generated $\hat R$-module. \\
 (B)
This is closely related to the notion of the canonical module $K_M$ of a finitely generated $R$-module 
$M.$ To this end we have to assume that $(R,\mathfrak m)$ is the epimorphic image of a local 
Gorenstein ring $(S,\mathfrak{n})$ with $n = \dim  S.$   
Then $K_M = \Ext_S^{n-d}(M,S)$ is called the canonical module of $M.$ By virtue of the Cohen Structure 
Theorem and  (A)  there is an isomorphism $K(M) \simeq K_{\hat{M}}.$ By 
the Local Duality Theorem (see \cite{G}) there are the following isomorphisms 
\[ 
\Hom_R(H^d_{\mathfrak{m}}(M), E_R(R/\mathfrak{m})) \simeq K(M) \simeq K_{\hat{M}}. 
\]
\end{definition} 

Moreover, in order to describe certain results on the canonical module we need a 
further technical definition. 

\begin{definition} \label{2.3} For an ideal $I \subset R$ with $\dim R/I = d$ we will denote 
by $I_d$ the intersection of those primary components in a minimal reduced primary 
decomposition of $I$ which are of dimension $d.$ If $Z \subset \Spec R$ and $d \in \mathbb{N},$ 
then we put $Z_d = \{\mathfrak{p} \in Z | \dim R/\mathfrak{p}= d\}.$ 
\end{definition} 

In the following we summarize a few results about the canonical module $K_M$ as well as 
of $K(M).$ Most of the statements are known in the literature. When we write $K_M$ we always 
assume that $R$ is a factor ring of a Gorenstein ring $S$ as above.  

\begin{proposition} \label{2.4} 
Let $M$ denote a finitely generated $R$-module. With the previous conventions and notation 
the following statements  hold: 
\begin{itemize} 
\item[(a)] $\Ass_{\hat{R}} K(M) = (\Ass_{\hat{R}} \hat{M})_d$ and $\Ass_R K_M = (\Ass_R M)_d.$ 
\item[(b)] $\Ann_{\hat{R}} K(M) = (\Ann_{\hat{R}} \hat{M})_d$ and $\Ann_R K_M = (\Ann_R M)_d.$ 
\item[(c)]  The module $K_M$ satisfies Serre's condition $S_2,$ that is 
\[ 
\depth (K_M)_{\mathfrak{p} }\geq \min \{2, \dim (K_M)_{\mathfrak{p}}\}  \text{ for all } \mathfrak{p} \in   
\Supp_R K_M.
\]
\item[(d)] The kernel of the natural map $R \to \Hom_R(K_R, K_R)$ is $0_d.$ It is an isomorphism 
if and only if  $R$ satisfies $S_2.$ 
\end{itemize}
\end{proposition} 

\begin{proof} For the proof of (a) and (b) we refer to \cite[Proposition 6.6]{G}. For the proof of 
(c) see \cite[Lemma 1.9]{HH} and for (d)  see \cite{A} and \cite[Proposition 1.2]{A-G},. For some additional information we refer also to \cite[Lemma 1.9]{Sch}.
\end{proof}

\section{Remarks to the Hartshorne-Lichtenbaum Vanishing Theorem}

 In this section let $M$ denote a $d$-dimensional, finitely generated $R$-module. Here 
 $(R,\mathfrak{m})$ is a local ring. The functor $\hat{\cdot}$ denotes the completion 
 functor. 
 
 For an $R$-module $M$ let $0 = \cap_{i = 1}^{n} Q_i(M)$ denote a reduced minimal primary 
 decomposition of the zero submodule of $M.$ That is $M/Q_i(M), i = 1,\ldots,n,$ is a $\mathfrak{p}_i$-primary 
 $R$-module.  Clearly $\Ass_R M = \{\mathfrak{p}_1,\ldots,\mathfrak{p}_n\}.$ 
 
 \begin{definition} \label{3.1}  Let $I \subset R$ denote an ideal of $R.$ We define two 
 disjoint subsets $U, V$ of $\Ass_R M$ related to $I.$
 \begin{itemize}
  \item[(a)] $U = \{\mathfrak{p} \in \Ass_R M | \dim R/\mathfrak{p} = d \text{ and } 
 \dim R/I +\mathfrak{p} = 0 \}.$ 
 \item[(b)] $V = \{\mathfrak{p} \in \Ass_R M | \dim R/\mathfrak{p} < d \text{ or } 
 \dim R/\mathfrak{p} = d \text{  and } \dim R/I + \mathfrak{p} > 0 \}.$ 
 \end{itemize}
 Finally we define $Q_I(M) = \cap_{\mathfrak{p}_i \in U} Q_i(M).$ In the case $U = \emptyset$ put 
 $Q_I(M) = M.$
 \end{definition}
 
 The following Lemma gives a better understanding of the previous definitions.
 
 \begin{lemma} \label{3.2} With the previous notation we have that 
 \[
 \Ass_R Q_I(M) = V, \; \Ass_R M/Q_I(M) = U\; \text{ and } \;U \cup V = \Ass_R M.
 \]
 \end{lemma} 
 
 \begin{proof} The proof is an easy consequence of the primary decomposition of $M, M/Q_I(M)$ 
 and $Q_I(M)$ (see \cite[Lemma 2.7]{Sch2}). 
 \end{proof}
 
 Now we are prepared in order to establish the first main result of this section. It explains 
 in more detail the structure of $H^d_I(M), d = \dim M.$ 
 
 \begin{theorem} \label{3.3} Let $I$ denote an ideal of a local ring $(R,\mathfrak{m}).$ Let $M$ be a finitely 
 generated $R$-module and $d = \dim M.$ Then there is a natural isomorphism 
 \[
 H^d_I(M) \simeq H^d_{\mathfrak{m}\hat{R}}(\hat{M}/Q_{I\hat{R}}(\hat{M})),
 \]
 where $\hat{\cdot}$ denotes the $\mathfrak{m}$-adic completion. 
 \end{theorem}
 
 \begin{proof} First note that $H^d_I(M)$ is an Artinian $R$-module. So it admits a unique  
 $\hat{R}$-module structure compatible with its $R$-module structure such that the natural 
 homomorphism 
 \[
 H^d_{I\hat{R}}(\hat{M}) \simeq H^d_I(M) \otimes_R \hat{R} \to H^d_I(M) 
 \] 
 is an isomorphism. That is, without loss of generality we may assume that $R$ is 
 complete. 
 
 Now apply the local cohomology functor to the short exact sequence
 \[
 0 \to Q_I(M) \to M \to M/Q_I(M) \to 0
 \] 
 it implies a natural isomorphism $H^d_I(M) \simeq H^d_I(M/Q_I(M)).$ To this end recall 
 that $H^i_I(Q_I(M)) = 0$ for all $i \geq d.$ The vanishing for $i = d$ follows by the 
 Hartshorne-Lichtenbaum Vanishing Theorem because of $\Ass_R Q = V,$ where $Q = Q_I(M).$ By 
 the base change of local cohomology there is the isomorphism 
 \[
 H^d_I(M/Q_I(M)) \simeq H^d_{I + \Ann_R M/Q}(M/Q).
 \]
 In order to complete the proof it is enough to show that $\mathfrak{m} = \Rad (I + \Ann_R M/Q).$ 
 To this end consider 
 \[
 V(I + \Ann_R M/Q) = V(I) \cup \Supp_R M/Q = \cup_{\mathfrak{p} \in U} V(I + \mathfrak{p}) = \{\mathfrak m\},
 \]
 as required.
 \end{proof}

 In the case of $M = R$ in Theorem \ref{3.3} it follow that 
 $H^d_I(R) = H^d_{\mathfrak{m} \hat{R}}(\hat{R}/Q_{I\hat{R}}(\hat{R}) ).$ 
 By the definition $Q_{I\hat{R}}(\hat{R}) $ is equal to the intersection of all the 
 $\mathfrak{p}$-primary component of a reduced minimal primary decomposition of the 
 zero ideal in $\hat{R}$ such that $\dim \hat{R}/\mathfrak{p} = \dim R$ and 
 $\dim \hat{R}/I \hat{R} + \mathfrak{p} = 0.$ 
 Next we want to extend this to the case of an $R$-module $M.$ 
 
 \begin{definition} \label{3.4} Let $M$ denote a finitely generated module over the  local 
 ring $(R, \mathfrak{m}).$   Let $I \subset R$ denote an ideal. Then define $P_I(M)$ as the 
 intersection of all the primary components of $\Ann_R M$ such that $\dim R/\mathfrak{p} 
 = \dim M$ and $\dim R/I + \mathfrak{p} = 0.$ Clearly $P_I(M)$ is the preimage of 
 $Q_{I R/\Ann_R M}(R/\Ann_R M)$ in $R.$ 
 \end{definition} 

 With these preparations we are able to prove the extension we have in mind. 
  
 \begin{corollary} \label{3.5} Let $M$ denote a finitely generated $R$-module and $d = \dim M.$ Let 
 $I \subset R$ be an ideal. 
 Then 
 \[
  H^d_I(M) \simeq H^d_{\mathfrak{m}\hat{R}}(\hat{M}/P_I(\hat{M})\hat{M}),
 \]
 where $P_I(\hat{M}) \subset \hat{R}$ is the ideal as defined in Definition \ref{3.4}.
 \end{corollary} 
 
 \begin{proof}  As in the beginning of proof  of Theorem \ref{3.3}  we may assume  that $R$ 
 is a complete local ring without loss of 
 generality. Let $\overline{R} = R/\Ann_R M.$ Then by base change 
 and the right exactness there are the isomorphisms 
 \[
 H^d_I(M) \simeq H^d_{I\overline{R}}(M) \simeq H^d_{I \overline{R}}(\overline{R}) \otimes_R M.
 \] 
 Now by virtue of Theorem \ref{3.3} there is the isomorphism $H^d_{I\overline{R}}(\overline{R}) \simeq 
 H^d_{\mathfrak{m}}(R/P_I(M)).$ Therefore it follows that 
 \[
 H^d_{I \overline{R}}(\overline{R}) \otimes_R M \simeq H^d_{\mathfrak{m}}(R/P_I(M)) \otimes_R M \simeq 
 H^d_{\mathfrak{m}}(M/P_I(M)M),
 \]
 which finishes the proof of the statement.
 \end{proof}

 Let $A$ denote an Artinian $R$-module. Then the decreasing sequence of submoduls 
 $\{\mathfrak{m}^nA\}_{n \in \mathbb{N}}$ becomes stable. Let $\langle \mathfrak{m} \rangle A$ 
 denote the ultimative stable value of this sequence of decreasing submodules. 
 
 \begin{remark} \label{3.6} Let $I \subset R$ denote an ideal. For a finitely generated 
 $R$-module $M$ there is a natural epimorphism 
 \[
 H^d_{\mathfrak{m}}(M) \to H^d_I(M) \to 0, \; d = \dim M,
 \]
 (see \cite{D-Sch}). In fact (see \cite[Theorem 1.1]{D-Sch}) it induces an isomorphism 
 \[
 H^d_I(M) \simeq H^d_{\mathfrak{m}}(M)/\sum_{n \in \mathbb{N}}\langle \mathfrak{m} 
 \rangle (0 :_{H^d_{\mathfrak{m}}(M)} I^n).
 \]
 Thus the kernel is described as 
 $\sum_{n \in \mathbb{N}} \langle \mathfrak{m} \rangle  (0 :_{H^d_{\mathfrak{m}}(M)} I^n).$ 
 
 Let us consider the previous epimorphism as an epimorphism of $\hat{R}$-modules. Then by 
 Corollary \ref{3.5} its kernel is equal to $P_I(\hat{M}) H^d_{\mathfrak{m}\hat{R}}(\hat{M}),$ 
 or in other words  
 \[
 H^d_I(M) \simeq H^d_{\mathfrak{m}\hat{R}}(\hat{M})/P_I(\hat{M}) H^d_{\mathfrak{m}\hat{R}}(\hat{M}).
 \] 
 This follows easily since $H^d_{\mathfrak{m}\hat{R}}(\hat{M}/P_I(\hat{M})\hat{M}) \simeq 
 H^d_{\mathfrak{m}\hat{R}}(\hat{M}) \otimes_{\hat{R}} \hat{R}/P_I(\hat{M}).$ 
 \end{remark} 
 
 By virtue of Corollary \ref{3.5} and Remark \ref{3.6} this proves Theorem \ref{1.1}.

 \section{On the endomorphism ring} 
 In this Section let $(R, \mathfrak{m})$ denote a $d$-dimensional local ring. For an ideal 
 $I \subset R$ we investigate the endomorphism ring of $H^d_I(R).$ It is non-zero if and only 
 if $H^d_I(R)$ fails the Hartshorne-Lichtenbaum Vanishing Theorem.  That is, we study the natural 
 homomorphism 
 \[
 R \to \Hom_R(H^d_I(R),H^d_I(R)), \quad r \mapsto m_r,
 \]
 where $m_r$ denotes the multiplication map by $r \in R.$ Since $H^d_I(R)$ admits 
 the structure of an $\hat{R}$-module (see \ref{2.1}) it follows that $\Hom_R(H^d_I(R),H^d_I(R))$ 
 has a unique natural $\hat{R}$-module such that the diagram 
 \[
 \begin{array}{ccc}
   R & \to & \Hom_R(H^d_I(R), H^d_I(R)) \\
  \downarrow &  & \Vert \\
   \hat R & \to & \Hom_{\hat R}(H^d_I(R), H^d_I(R)).
 \end{array}
 \] 
is commutative. 
 That is, the map $R \to \Hom_R(H^d_I(R),H^d_I(R))$ factors through $\hat{R}.$ Before we study 
 the endomorphism ring we need an auxiliary statement on the Matlis dual of $H^d_I(R).$ 
 
 \begin{lemma} \label{4.1} Let $I$ denote an ideal in a local ring $(R, \mathfrak{m})$ and 
 $d = \dim R.$ 
 \begin{itemize} 
 \item[(a)] 
 $T_I(R) = \Hom_R(H^d_I(R),E_R(R/\mathfrak{m}))$ is a finitely generated $\hat{R}$-module.   
 \item[(b)]  $\Ass_{\hat{R}} T_I(R) = \{\mathfrak{p} \in \Ass \hat{R} | 
 \dim \hat{R}/\mathfrak{p} = \dim R \text{ and } \dim \hat{R}/I \hat{R} + \mathfrak{p} = 0 \}.$
 \item[(c)] $K_{\hat{R}}(\hat{R}/Q_I(\hat{R})) \simeq T_I(R).$
 \end{itemize}
 \end{lemma} 

\begin{proof} The statements follow by the definition and the auxiliary 
results above. 
\end{proof}
 
For an $R$-module $M$ the natural map $R \to \Hom_R(M,M)$ is in general neither injective 
nor surjective. For the local cohomology module $H^d_I(R)$ we get a more precise picture.  
 
 \begin{theorem} \label{4.2} Let $I$ denote an ideal in a local ring $(R,\mathfrak{m})$ 
 with $d = \dim R.$ Let 
 \[
 \Phi: \hat{R} \to \Hom_{\hat{R}}(H^d_I(R), H^d_I(R)) 
 \]
 the natural homomorphism, where $H^d_I(R)$ is as an Artinian $R$-module considered as 
 an $\hat{R}$-module.
 \begin{itemize}
 \item[(a)] $\ker \Phi = Q_{I\hat{R}}(\hat{R}).$ 
 \item[(b)] $\Phi$ is surjective if and only if $\hat{R}/Q_{I\hat{R}}(\hat{R})$ satisfies $S_2.$ 
 \item[(c)] $\Hom_{\hat{R}}(H^d_I(R), H^d_I(R))$ is a finitely generated $\hat{R}$-module. 
 \item[(d)] $\Hom_{\hat{R}}(H^d_I(R), H^d_I(R))$ is a commutative semi-local Noetherian ring.
 \end{itemize}
 \end{theorem}  
 
 \begin{proof} First note that as $H^d_I(R)$ is an Artinian $R$-module so $H^d_I(R) \simeq 
 H^d_{I\hat{R}}(\hat{R})$ (see \ref{2.1} for the detail). That is, without loss of generality we may 
 assume that $R$ is a complete local ring.  By virtue of Theorem \ref{3.3} there is the natural 
 isomorphism $H^d_I(R) \simeq H^d_{\mathfrak{m}}(R/Q), Q = Q_I(R).$ Then 
 \[
 K_{R/Q} \simeq D(H^d_{\mathfrak{m}} (R/Q)) \simeq \Hom_R(H^d_I(R), E_R(R/\mathfrak{m})).
 \] 
 Because $H^d_I(R)$ is Artinian the Matlis' duality provides an isomorphism 
 \[
 \Hom_R(H^d_I(R), H^d_I(R)) \simeq \Hom_R(K_{R/Q},K_{R/Q}).
 \]
 Therefore the kernel of $\Phi$ equals to $\Ann_R K_{R/Q} = Q_d,$ which proves (a). Because the 
 endomorphism ring of $H^d_I(R)$ is isomorphic to the endomorphism ring of 
 the canonical module of $K_{R/Q}$ the results in (b), (c) and (d) are shown by Aoyama 
 (see \cite[Proposition 1.2]{A}), Aoyama and 
 G\^{o}to (see \cite[Theorem 3.2]{A-G}) and Hochster and Huneke (see \cite[(2.2)]{HH}).
 \end{proof} 

 Note that the ideal $J \subset R$ as considered in Theorem \ref{0.1} (b) in the case of a complete local ring 
 $(R,\mathfrak{m})$ coincides  with $Q_I(R)$ in Theorem \ref{4.2}.
  In the next we want to relate some homological properties of $T_I(R)$ with those of 
 the endomorphism ring $\Hom_R(H^d_I(R), H^d_I(R))$ resp. $\hat{R}/Q_{I\hat{R}}(\hat{R}).$  
 
 \begin{theorem} \label{4.3} Let $I$ be an ideal in $(R,\mathfrak m),$ a complete local ring and $\dim R = d.$ 
 For an integer $r \geq 2$ we have the following statements: 
 \begin{itemize} 
 \item[(a)] Suppose $R/Q_I(R)$ has $S_2.$ Then $T_I(R)$ satisfies the condition $S_r$ if and only if
  $H^i_{\mathfrak{m}}(R/Q_I(R)) = 0$ for $d-r+2 \leq i < d.$ 
 \item[(b)] $R/Q_I(R)$ satisfies the condition $S_r$ if and only if $H^i_{\mathfrak{m}}(T_I(R)) = 0$ for 
 $d-r+2 \leq i < d$ and $R/Q_I(R) \simeq \Hom_R(H^d_I(R), H^d_I(R)).$  
 \end{itemize} 
 In particular, if $R/Q_I(R)$ has $S_2$ it is a Cohen-Macaulay ring if and only if the module $T_I(R)$ is Cohen-      Macaulay.
 \end{theorem} 
 
\begin{proof} By our conventions and definitions it 
follows that $T_I(R) \simeq K_{R/Q},$ where $Q = Q_I(R),$ and $ R/Q \simeq \Hom_R(H^d_I(R), H^d_I(R))$.
Then the statement in (a) resp. in (b) follows by virtue of \cite[1.14]{Sch} for $M = R/Q$ resp. $M= K_{R/Q}.$
\end{proof}
 
 \section{On Connectedness and Indecomposability}
 The next part of our investigations is to characterize the number of the maximal ideals of the 
 endomorphism ring $\Hom_{\hat{R}}(H^d_I(R), H^d_I(R)), d = \dim R$ (see \ref{4.2}). To this end we need 
 a few more preparations. First we recall a definition given by Hochster and Huneke (see \cite[(3.4)]{HH}). 
 
 \begin{definition} \label{5.1} Let $(R,\mathfrak{m})$ denote a local ring. 
 We denote by $\mathbb{G}(R)$ the undirected graph whose vertices are primes
$\mathfrak{p} \in \Spec R$ such that $\dim R = \dim R/\mathfrak{p},$ and
two distinct vertices $\mathfrak{p}, \mathfrak{q}$ are joined by
an edge if and only if $(\mathfrak{p}, \mathfrak{q})$ is an ideal of
height one. 
 \end{definition} 
  
  Next we are interested in the connectedness of $\mathbb{G}(R).$ That  is characterized in 
  the following statement. To this end we refer to the notion of connectedness in codimension one 
  of $\Spec R$ as defined by Hartshorne (see \cite{rH}). 
  
 \begin{proposition} \label{5.2} Let $(R, \mathfrak{m})$ denote a local ring with $d = dim R.$ 
 Then the following conditions are equivalent:
 \begin{itemize} 
 \item[(i)] The graph $\mathbb{G}(R)$ is connected. 
 \item[(ii)] $\Spec R/0_d$ is connected in codimension one.
 \item[(iii)] For every ideal $J R/0_d$ of height at least two, $\Spec (R/0_d) \setminus V(JR/0_d)$ is 
 connected.
 \end{itemize} 
 \end{proposition}  
 
 \begin{proof} By the definitions (see \ref{5.1} and \cite{rH}) this is easily seen. See also 
 \cite[(3.6)]{HH}. 
 \end{proof}
 
 Next we describe when the endomorphism ring of $H^d_I(R), d = \dim R,$ is a local ring. 
 We call an $R$-module $X$ indecomposable if it is not the direct sum of two non-trivial 
 submodules.  
  
 \begin{theorem} \label{5.3} Let $(R, \mathfrak{m})$ denote a complete local ring and 
 $d = \dim R.$ For an ideal $I \subset R$ the following conditions are equivalent: 
 \begin{itemize} 
 \item[(i)] $H^d_I(R)$ is indecomposable. 
 \item[(ii)] $T_I(R)$ is indecomposable.
 \item[(iii)] The endomorphism ring of $H^d_I(R)$ is a local ring. 
 \item[(iv)] The graph $\mathbb{G}(R/Q_I(R))$ is connected.
 \end{itemize} 
 \end{theorem} 
 
 \begin{proof} We may always assume that $Q = Q_I(R)$ is a proper ideal. In the case of $Q = R$ 
 there is nothing to show. As it follows by  above investigations we have the following isomorphisms 
 \[
 H^d_{\mathfrak{m}}(R/Q) \simeq H^d_I(R), K_{R/Q} \simeq T_I(R) \text{ and } \End H^d_{\mathfrak{m}}(R/Q) 
 \simeq \End H^d_I(R),
 \] 
 where  $\End$ denotes the endomorphism ring. That is, we have reduced the 
 proof of the statement to the corresponding result for $H^d_{\mathfrak{m}}(R/Q).$ Note that 
 $d = \dim R/Q.$ Then the equivalence of the conditions is proved by Hochster and Huneke (see 
 \cite[(3.6)]{HH}). 
 \end{proof}

 Now we shall describe $t,$ the number of connected components of  $\mathbb{G}(R/Q_I(R)).$ 
 
 \begin{definition} \label{5.4} Let $I$ be an ideal in a local ring $(R,\mathfrak{m}).$ Suppose that  $Q = Q_I(R)$ is  
 a proper ideal. Let $ \mathbb{G}_i, i = 1,\ldots,t,$ denote the connected components of 
 $ \mathbb{G}(R/Q).$  Let $Q_i, i = 1,\ldots,t,$ denote the intersection of all $ \mathfrak{p}$-primary 
 components of a reduced minimal primary decomposition of $Q$ such that $ \mathfrak{p} \in 
 \mathbb{G}_i.$ Then $Q = \cap_{i=1}^t Q_i$ and $ \mathbb{G}(R/Q_i) = \mathbb{G}_i, i = 1,\ldots,t,$ is 
 connected. Moreover, let $I_i, i = 1,\ldots,t,$ denote the image of the ideal $I$ in $R/Q_i.$ 
 \end{definition}
 
 \begin{theorem} \label{5.5} Let $I$ denote an ideal of a complete local ring $(R, \mathfrak{m})$ 
 with $d = \dim R \geq 2.$ Then 
\[ 
 \End H^d_I(R) \simeq \End H^d_{I_1}(R/Q_1) \times \ldots \times \End H^d_{I_t}(R/Q_t)
 \] 
 is a semi-local ring, $\End H^d_{I_i}(R/Q_i), i = 1,\ldots,t,$ is a local ring and therefore 
 $t$ is equal  to the number of maximal ideals of $\End H^d_I(R).$
 \end{theorem}   
 
 \begin{proof} As in the proof in Theorem \ref{5.3} we have $\End H^d_{\mathfrak{m}}(R/Q) 
 \simeq \End H^d_I(R).$ For an integer $1 \leq i \leq t$ we define $ \tilde{Q}_i = \cap_{j=1}^i Q_j,$ 
 in particular $ \tilde{Q}_t = Q.$ Then there is the short exact sequence 
 \[
 0 \to R/\tilde{Q}_{i+1} \to R/\tilde{Q}_i \oplus R/Q_{i+1} \to R/(\tilde{Q}_i + Q_{i+1}) \to 0.
 \]
 Because $ \mathbb{G}_{i+1}$ and $ \mathbb{G}_j$ for $j = 1,\ldots,i,$ are not connected it 
 follows by the definition that $\height(\tilde{Q}_i + Q_{i+1}) \geq 2$ and therefore 
 $\dim R/(\tilde{Q}_i + Q_{i+1}) \leq d-2.$ Whence the short exact sequence induces 
 isomorphisms $H^d_I(R/\tilde{Q}_{i+1}) \simeq H^d_I(R/\tilde{Q}_i) \oplus H^d_I(R/Q_{i+1})$  
 and by induction 
\[ 
 H^d_I(R/Q) \simeq \oplus_{i=1}^t H^d_I(R/Q_i) .
 \] 
 Furthermore, because of Theorem \ref{3.3} and Corollary \ref{3.5} we have 
\[  
  H^d_I(R) \simeq H^d_{\mathfrak{m}}(R/Q)  
  \text{ and } H^d_{I_i}(R/Q_i)  \simeq H^d_I(R/Q_i)  \simeq H^d_{\mathfrak{m}}(R/Q_i), i = 1,\ldots,t. 
  \]
  Now by Matlis duality it turns out that 
  \[
  \End H^d_I(R) \simeq \End K_{R/Q} \text{ and } \Hom_R(H^d_{\mathfrak{m}}(R/Q_j), H^d_{\mathfrak{m}}(R/Q_i)) 
  \simeq \Hom_R(K_{R/Q_i}, K_{R/Q_j})
  \] 
  for all $i,j = 1,\ldots,t.$ Moreover we see that $ \Hom_R(K_{R/Q_i}, K_{R/Q_j}) = 0$ for $i \not= j$ because 
  \[
  \Ass_R \Hom_R(K_{R/Q_i}, K_{R/Q_j}) = \Ass_R K_{R/Q_j} \cap \Supp_R R/Q_i = \emptyset 
  \] 
  for all $i \not= j$ as follows by the definitions and  Proposition \ref{2.4}.  This implies the 
  decomposition 
  \[
  \End H^d_I(R)  \simeq \End H^d_{I_1}(R/Q_1) \times \ldots \times \End H^d_{I_t}(R/Q_t)
  \]
  because $\End K_{R/Q_i} \simeq \End H^d_{I_i}(R/Q_i), i = 1,\ldots,t,$ as follows again by Matlis 
  duality. By Theorem \ref{5.3} the endomorphism ring of $H^d_{I_i}(R/Q_i), i = 1,\ldots, t,$ is a
  local ring. So we get the decomposition as a direct product of rings and $\End H^d_I(R)$ is a 
  semi-local ring with $t$ as its number of maximal ideals. 
  \end{proof} 
  
{\bf Acknowledgement.} The authors are grateful to the reviewer for suggesting several improvements 
of the manuscript.


\end{document}